\documentclass[11pt]{amsart}
\usepackage{amscd}
\usepackage[arrow,matrix]{xy}
\input xypic
\usepackage{graphicx, tikz}
\usepackage{hyperref}
\usepackage{soul}
\usepackage{comment}
\usepackage{amsmath}
\usepackage{amsmath, latexsym, amssymb}
\numberwithin{equation}{section}
\theoremstyle{plain}

\newtheorem{lemma}{Lemma}[section]
\newtheorem{proposition}[lemma]{Proposition}
\newtheorem{theorem}[lemma]{Theorem}
\newtheorem{corollary}[lemma]{Corollary}
\newtheorem{problem}[lemma]{Problem}

\theoremstyle{definition}
\newtheorem{definition}[lemma]{Definition}
\newtheorem*{definition*}{Definition}
\newtheorem{remark}[lemma]{Remark}
\newtheorem{example}[lemma]{Example}

\usepackage{color}
\usepackage{cancel}

\definecolor{brown}{RGB}{150,100,0}

\definecolor{purple}{RGB}{150,0,100}
\definecolor{grey}{RGB}{128,128,128}

\newcommand{\R}{{\mathbb R}}

\newcommand{\N}{{\mathbb N}}

\newcommand{\Ff}{{\mathcal F}}
\newcommand{\Gg}{{\mathcal G}} %gauge transformations
\newcommand{\Hh}{{\mathcal H}}

\newcommand{\Mm}{{\mathcal M}} %moduli space
\newcommand{\Nn}{{\mathcal N}}

\newcommand{\Pp}{{\mathcal P}}

\newcommand{\Ss}{{\mathcal S}}

\newcommand{\Xx}{{\mathcal X}}
\newcommand{\Yy}{{\mathcal Y}}
\newcommand{\Zz}{{\mathcal Z}}

\mathchardef\mhyp="2D

\newcommand{\eps}{{\varepsilon}}

\newcommand{\Meas}{{\bf Meas}}
\newcommand{\Probm}{{\bf Probm}}

\def\NABLA#1{{\mathop{\nabla\kern-.5ex\lower1ex\hbox{$#1$}}}}
\def\Nabla#1{\nabla\kern-.5ex{}_#1}

\mathchardef\mhyp="2D

\DeclareMathOperator{\Id}{Id}

\newcommand{\pb}{{\mathbf p}}
\newcommand{\qb}{{\mathbf q}}

\begin{document}
\title[Probabilistic morphisms]{Probabilistic morphisms and Bayesian supervised learning}
\author[H. V. L\^e]{H\^ong V\^an L\^e}	
\address{Institute of Mathematics of the Czech Academy of Sciences, Zitna 25, 11567 Praha 1, Czech Republic $\&$  Charles University, Faculty of Mathematics and Physics
	Ke Karlovu 3, 121 16 Praha 2, Czech  Republic}

\email{hvle@math.cas.cz}

%\thanks{Research of HVL was supported by the Institute of Mathematics,  Czech Academy of Sciences (RVO: 67985840) and GA\v CR-project 22-00091S}
	\date{\today}

\begin{abstract}    In this  paper,     we  develop category  theory of Markov kernels to the  study  of categorical  aspects  of Bayesian inversions.    As a  result,  we present a unified model for Bayesian  supervised learning, encompassing Bayesian density estimation. We   illustrate this model with  Gaussian process regressions. 

\end{abstract}

\maketitle
%\tableofcontents

%%%%%%%%%%%%%%%%%%%%%%%%%%%%%%

\section{Preface}
I dedicate this paper to my teacher, Professor Anatoly Fomenko, on  his 80th birthday. This work is inspired by Professor Fomenko's contributions to various fields of mathematics and his passion for exploring the connections between mathematics and natural sciences. I wish him continued good health and great success in the years to come.

\section{Introduction}\label{sec:intro}

     For a measurable space $\Xx$,  we  denote  by $\Sigma_\Xx$  the $\sigma$-algebra of $\Xx$,  and by $\Ss(\Xx), \Mm (\Xx), \Pp (\Xx)$ the  spaces  of  all signed finite  measures,   (nonnegative) finite  measures,  probability measures, respectively,   on $\Xx$.  If $\Xx$ is a  topological space, we  consider the Borel $\sigma$-algebra on $\Xx$, unless otherwise  stated.
     In this paper  we are interested in solving the following  problem.
 
\begin{problem}[{\bf Supervised  Inference  Problem SBI}]\label{prob:transuper} Let $\Xx$  be  an input  space and  $\Yy$    a measurable label  space.  Given  training  data
	$S_n: =\big ((x_1, y_1), \ldots ,\\ (x_n, y_n)\big) \in  (\Xx \times \Yy)^n$  and   new  test data $T_m :  = (t_1, \ldots, t_m)\in \Xx^m$,  estimate  the  {\it predictive   probability  measure}  in $ \Pp (\Yy^m)$  that  governs  the joint distribution of  the   $m$-tuple  $ \big(y_1, \ldots, y_m\big) \in \Yy^m$  where $y_m$  is the label of $t_m$.   
\end{problem} 

If $\Xx$  consists  of a single point,   Problem SBI  is equivalent  to   the   fundamental problem of  density/probability estimation in classical mathematical statistics.

 A  non-Bayesian  (frequentist) approach to Problem SBI    begins  with   the following    assumption:  the labeled    items   $(x_i, y_i)\in  (\Xx\times  \Yy), \, i \in [1,n],$ are identically independently  distributed  (i.i.d.) by an  (unknown) probability measure  $\mu \in \Pp (\Xx\times \Yy)$. In particular, $\Xx$ must be a measurable space. Under this assumption   and other mild conditions, e.g.,  $\Yy$ is a standard Borel space, or  $\Xx$ is a countable sample space, or $\mu$ is dominated by  a product measure $\nu_0\otimes \rho_0$ where $\nu_0$ and $\rho_0$ are $\sigma$-finite measures on $\Xx$ and $\Yy$, respectively,   the  joint distribution  of the    $m$-tuple  $(y_1, \ldots, y_m) \in \Yy^m$ equals  $\otimes _{i =1}^m \mu (\cdot| t_i) \in \Pp (\Yy^m)$  where $\mu (\cdot|t_i)\in \Pp (\Yy)$ is  the    probability  measure  of the label   of $t_i\in \Xx$,  which  is a  regular  version of  conditional  probability  measure  $\mu$  with respect to the projection $\Pi_\Xx: \Xx \times \Yy \to \Xx$, conditioning    at $t_i$.     Within the  non-Bayesian  framework,  to solve  Problem SBI  one  has  to  approximate the  {\it supervisor  operator}  $\mu_{\Yy|\Xx}: \Xx \to \Pp (\Yy)$, which represents  the conditional  probability  measure  $\mu$  with respect to the projection $\Pi_\Xx: \Xx \times \Yy \to \Xx$,    by an optimal  supervisor $h_{opt}$ in   a given hypothesis   class  $\Hh$ of  possible  supervisors, and then evaluate  the optimal supervisor  $h_{opt}\in \Hh$ at  a  test point $t_i$.  Note that $\Hh$ is often chosen to be a    subset  of the     set $\Meas (\Xx, \Yy)$ of all  measurable mappings   from $\Xx$ to $\Yy$, identifying  $f(x)\in \Yy$ with   the Dirac measure  $\delta_{f (x)} \in \Pp (\Yy)$ \cite{Vapnik1998}.

The i.i.d. assumption  is not always      satisfied in life. In Bayesian  statistics, one   replaces  the i.i.d.  assumption by   the weaker  condition of conditionally  i.i.d. assumption \cite{Schervish1997}. In this paper, using the category of probabilistic morphisms, we propose a  Bayesian  approach to  solve Problem SBI, which encompasses 
 Bayesian  models for  regression  learning and density estimation.

The plan of our paper is   as follows. In Section \ref{sec:unified}, we  present  basic  concepts  of  the category of Markov kernels,  also called probabilistic morphisms, which  has been introduced  independently  by Lawvere \cite{Lawvere1962} and  Chentsov \cite{Chentsov1965}.
   In particular,  we study a contravariant functor  from the  category of probabilistic morphisms to the category of Banach spaces  (Theorem \ref{thm:clbanach}) to show that a composition  of Bayesian inversions  is a  Bayesian inversion (Proposition \ref{prop:bayesgroup}).  Then, we introduce the concept  of the graph of a probabilistic morphism (Definition \ref{def:graph}), and  using  this concept  we give a characterization of a Bayesian inversion  (Theorem \ref{thm:marginal}, Lemma \ref{lem:bsuff}).   In Section \ref{sec:bsl},  we propose  our solution  to SBI Problem  by presenting a unified model of 
   Bayesian supervised learning (Definition \ref{def:fposteriord}, Proposition \ref{prop:univ}) and  illustrate  this models   in   Bayesian regression learning (Corollary \ref{cor:bnrp}, Definition \ref{def:bnr}, Example  \ref{ex:gauss}). In the  final section \ref{sec:final} we discuss our results  and some  related open problems.

\section{The category of probabilistic morphisms and Bayesian inversions}\label{sec:unified}
In this section,  we  first recall     Chentsov's  and Lawvere's  approaches    to    the  category of  Markov kernels (Propositions \ref{prop:chentsovbanach}, Remark \ref{rem:functors}, Theorem \ref{thm:clbanach}).   Then,  we  provide a characterization  of  regular  conditional  measures using  a concept  of  the graph  of a probabilistic  morphism (Definition \ref{def:graph}, Theorem \ref{thm:marginal}). Finally, we recall  the concept  of a Bayesian inversions  and show that the equivalence classes  of Bayesian  inversions   form a   groupoid (Proposition \ref{prop:bayesgroup}).

 \subsection{Notation and conventions}\label{subs:notation} 
 \
 
 $\bullet$ For  any    measurable  space $\Xx$,  we  denote by $\Sigma _w   $  the smallest $\sigma$-algebra  on $\Pp (\Xx)$ such that  for  any $A \in \Sigma _\Xx$  the function $I_{ 1_A}: \Pp (\Xx) \to \R, \mu \mapsto   \int 1_A  d\mu $ is  measurable.  Here $1_A$ is the  characteristic  function  of $A$.  In our paper  we always  consider  $\Pp (\Xx)$  as  a measurable   space  with    the   $\sigma$-algebra  $\Sigma_w$, unless  otherwise stated.

 $\bullet$  A Markov  kernel $T: \Xx \times  \Sigma_\Yy \to [0,1]$    is  uniquely defined   by    the map  $ \overline T:  \Xx \to  \Pp (\Yy)$  such that  $\overline T (x) (A) = T (x, A)$ for all $x\in \Xx, A \in \Sigma _\Yy$.   We  shall   also use notations $T ( A|x):= T (x, A) $ and $\overline T(A|x) : = \overline T (x)(A)$.  
 
 $\bullet$ A    {\it probabilistic  morphism}  $T: \Xx \leadsto \Yy$  is an  arrow assigned to a  measurable mapping, denoted  by $\overline T$,  from $ \Xx$ to  $\Pp (\Yy)$.  We say  that  $T$ is generated  by $\overline  T$.  For  a measurable mapping $T: \Xx \to \Pp (\Yy)$ we  denote  by $\underline T: \Xx \leadsto  \Yy$ the generated  probabilistic morphism.
 
 $\bullet$  For   probabilistic  morphisms  $T_{\Yy|\Xx} : \Xx \leadsto \Yy$ and  $T_{\Zz|\Yy}: \Yy \leadsto  \Zz$  their    composition is the  probabilistic  morphism
 $$  T_{\Zz|\Xx} : = T_{\Zz |\Yy} \circ  T _{\Yy|\Xx}: \Xx \leadsto \Zz  $$
 $$	(T_{\Zz|\Yy}\circ T_{\Yy|\Xx}) (x, C): = \int_\Yy T_{\Zz|\Yy} (y ,C) T_{\Yy| \Xx} (dy|x) $$
 for	$x\in \Xx$   and $C \in \Sigma_\Zz$. 
 It is well-known that the composition is associative.

 $\bullet$ We  denote  by $\Meas(\Xx, \Yy)$  the set  of  all measurable mappings from   a measurable  space  $\Xx$ to a measurable space $\Yy$, and by $\Probm (\Xx, \Yy)$  the set  of all  probabilistic  morphisms    from $\Xx$ to $\Yy$.
 We denote  by $\Yy ^\Xx$  the set  of all mappings  from $\Xx$ to $\Yy$.
 For  any $\Xx$ we denote by $\Id_\Xx$ the  identity map on $\Xx$. For a   product   space  $\Xx \times \Yy$ we denote  by $\Pi_\Xx$ the canonical projection   to the    factor  $\Xx$.
 
 $\bullet$ We denote by $\mathbf {Meas}$ the category of measurable  spaces, whose objects  are measurable  spaces   and  morphisms  are measurable mappings.

   % For $\mu \in \Pp  (\Xx\times \Yy)$ denote by $\mu_\Xx$  the marginal   probability measure  $(\Pi_\Xx)_* \mu\in \Pp (\Xx)$. 
 
 Important  examples  of  probabilistic morphisms  are   measurable  mappings  $\kappa: \Xx \to \Yy$, since    the  Dirac map $\delta:\Xx \to \Pp (\Xx), x\mapsto \delta_x,$ is measurable \cite[Theorem 1]{Giry1982},  and hence  we  regard $\kappa$ as a  probabilistic morphism   which is generated  by the measurable mapping $\overline \kappa: =\delta \circ \kappa: \Xx \to \Pp (\Yy)$.  We  shall denote    measurable mappings by straight  arrows  and  probabilistic morphisms  by curved  arrows.  Hence  $\Meas (\Xx, \Yy) $  can be regarded as a  subset of $\Probm(\Xx, \Yy)$.  Other important  examples  of  probabilistic morphisms  are regular  conditional   probability measures,  whose   definition we    recall now.
 
   Let $\mu$  be   a finite measure on $(\Xx \times \Yy, \Sigma_\Xx \otimes \Sigma_\Yy)$.  A   {\it product regular  conditional  probability measure} \index{product  regular conditional measure} for $\mu$ with respect  to the  projection $\Pi_\Xx: \Xx \times \Yy \to \Xx$ is a  Markov  kernel\index{Markov kernel} $\mu_{\Yy|\Xx}: \Xx \times \Sigma_\Yy \to [0,1]$ such that
 \begin{equation}\label{eq:regpr}
 	\mu (A \times B) = \int_A  \mu_{\Yy|\Xx} (x, B)\, d (\Pi_\Xx)_* \mu  (x)
 \end{equation}
 for  any $A \in \Sigma_\Xx$ and $B \in \Sigma_\Yy$ \cite{LFR2004}. We shall    simply call $\mu_{\Yy|\Xx}$ a {\it regular  conditional  probability measure} for $\mu$,  and we identify $\mu_{\Yy|\Xx}$ with the  generating measurable  map $\overline{\mu_{\Yy|\Xx}}: \Xx \to \Pp(\Yy), x\mapsto \mu_{\Yy|\Xx} (\cdot |x)$.

 \subsection{Properties of Lawvere's  $\sigma$-algebra $\Sigma_w$}\label{subs:law} The $\sigma$-algebra $\Sigma_w$ was proposed by Lawvere in 1962  to   regard Markov kernels  as  measurable mappings.
 In  this subsection we  collect several  properties of    Lawvere's $\sigma$-algebra $\Sigma_w$ on $\Pp (\Xx)$  that shall be utilized in later parts  of this paper.
 \begin{proposition}\label{prop:lawvera}
 	(1) For any $k \in \N^+$ the multiplication mapping
 	\begin{equation}\label{eq:frakmk}
 		\mathfrak m^k:\prod_{i =1}^k\big (\Pp(\Xx_1) , \Sigma_w\big ) \to   \Big (\Pp \big (\prod_{i=1}^k \Xx_i\big ), \Sigma_w \Big ), \,  (\mu_1, \ldots, \mu_k) \mapsto \otimes_{i=1}^k \mu_i
 	\end{equation}
 	is  measurable.
 	
 %	(2) The map  $\delta: \Xx \to \Pp (\Xx)$ is measurable.
 	
 	(2)  Let $V$  be a  topological  vector space  and $\mu_0 \in \Pp (V)$.   
 	
 	(2a)   
 	 The  map  
 	$$I_{\mu_0}: \Pp (V) \to \Pp (V \times V),  \mu \mapsto \mu \otimes \mu_0$$
 	is measurable.
 	
 	(2b)   We define     the {\it convolution  operator} $* : \Pp (V)\times \Pp (V) \to \Pp  (V)$ by   the following  formula 
 	$$ \mu * \nu : =   (\mathrm{Ad})_*   (\mu \otimes \nu)$$
 	where $ \mathrm{Ad}: V\times V \to V,  (x, y) \mapsto x +y$, cf. \cite[Definition  3.9.8, p. 207, vol. 1]{Bogachev2007}.
 	Then the map 
 	$$ C_{\mu_0}: \Pp (V) \to \Pp (V),  \mu \mapsto  \mu\otimes \mu_0$$
 	is measurable.
 \end{proposition}

\begin{proof}   (1)	To prove the first assertion of  Proposition \ref{prop:lawvera}, using  induction argument,  it suffices to consider the case  $k=2$.  To  prove that the multiplication map  $\mathfrak m^2$ is measurable,  it suffices  to show that  for any $A , B \in \Sigma _\Xx$  the map $I_{1_{A\otimes B}}: \big(\Pp (\Xx)\times \Pp ( \Yy) , \Sigma_w \otimes\Sigma_w\big ) \to \R,\, (\mu, \nu) \mapsto \mu (A)\cdot\nu (B)$
	is measurable.  The  map  $I_{1_{A \otimes B}}$ is measurable  since  it can be written as  the composition of measurable  mappings
	\begin{equation}\label{eq:mult} 
		\big(\Ss (\Xx)\times \Ss ( \Yy) , \Sigma_w \otimes\Sigma_w\big ) \stackrel{(I_{1_A}\times I_{1_B})}{\longrightarrow}\R \times \R \stackrel{\mathfrak m_\R}{\longrightarrow \R}
	\end{equation}
	where   $\mathfrak m_\R  (x, y)  = x\cdot  y$.
	
%	(2) The assertion  (2) is well-known  \cite[Theorem 1]{Giry1982}.
	
	(2a)  The embedding $	I_{\mu_0}$
is a measurable  mapping since  $I_{\mu_0}$  is  the composition of the  following measurable mapping
	$$ \Pp (\Xx) \to  \Pp (\Xx)\times \Pp (\Yy), \mu \mapsto  (\mu, \mu_0)$$
	with the   multiplication map $\mathfrak m: \Pp (\Xx)\times \Pp (\Yy) \to \Pp (\Xx\times \Yy)$, which is measurable by   the  first assertion of Proposition \ref{prop:lawvera}.
	
(2b)	 The map  $C_{\mu_0}$ is measurable  since   $C_{\mu_0} =  (\mathrm{Ad})_* \circ I_{\mu_0}$, and     the measurability  of $(\mathrm{Ad})_*$ is   a consequence  of a  result  due to Giry \cite{Giry1982}, see  Remark \ref{rem:functors}(2) in the next  subsection,  taking into  account that $\mathrm{Ad}: V \times V \to V$ is a continuous, and hence, measurable map.
\end{proof}

\begin{remark}\label{rem:m}  For the general  abstract story behind  the formation of the map $\frak m ^2$, see \cite{Kock2011}, and see \cite{FPR2021}
	for a similar result.
\end{remark}
 \subsection{Graphs of probabilistic morphisms and characterizations  of regular  conditional  probability  measures} \label{subs:probm}
  
 \subsubsection{The category of  probabilistic morphisms}\label{ssub:probm}
 We denote by $\mathbf{Probm}$  the category  whose objects  are measurable  spaces   and  morphisms  are probabilistic morphisms  $ T:  \Xx\leadsto\Yy$.  The identity morphism  is $\Id _\Xx$.   For $\kappa \in \Meas(\Xx, \Yy)$ we  also use the shorthand  notation
 \begin{equation}\label{eq:delta}
 	\overline{\kappa}: = \delta \circ  \kappa.
 \end{equation}

  For any $T \in \Probm (\Xx, \Yy)$, $\mu \in \Ss (\Xx)$  and $B \in \Sigma_\Yy$ we set
 \begin{equation}\label{eq:Mhomomorphism}
 	S(T)_*\mu  (B) = \int _\Xx  \overline T(B|x) \, d\mu (x).
 \end{equation}
Chentsov  proved that the category 
$\Probm$ admits  a faithful  functor $S$ to the category  $\mathbf {Ban}$  of Banach  spaces   whose  morphisms  are bounded  linear mappings  of operator norm less than or equal one.
 
 \begin{proposition}\label{prop:chentsovbanach} \cite[Lemmas 5.9, 5.10]{Chentsov72}  Let  $S$  assign  to each  measurable space  $\Xx$ the   Banach  space $\Ss(\Xx)$  of  finite  signed measures    on $\Xx$ endowed  with the  total variation norm  $\| \cdot \|_{TV}$ and  to every  Markov kernel  $T_{\Yy|\Xx}$ the Markov homomorphism  $(T_{\Yy|\Xx})_*$.
 	Then  $S$ is a faithful  functor from  the category $\mathbf {Probm}$ to the category $\mathbf {Ban}$.	
 \end{proposition}

\begin{remark}\label{rem:functors} (1) It is known that the restriction $M_*(T)$  of $S_*(T)$ to  $\Mm(\Xx)$  and the  restriction $P_*(T)$  of $S_*(T)$  to  $\Pp (\Xx)$ maps  $\Mm(\Xx)$ to  $\Mm(\Yy)$ and $\Pp (\Xx)$ to
	$\Pp(\Yy)$, respectively  \cite[Lemma 5.9, p. 72]{Chentsov72}.  We   shall use   the shorthand notation $T_*$  for $S_* (T)$, $M_* (T)$ and $P_* (T)$, if no misunderstanding occurs.  
	
	(2) Let $T \in \Probm(\Xx, \Yy)$. Then  $T_*: \big(\Pp(\Xx),\Sigma_w\big) \to \big(\Pp(\Yy),\Sigma_w\big)$ is a measurable   mapping \cite{Lawvere1962}, \cite[Theorem 1]{Giry1982}. We shall  call the functor $P: \Probm \to \Meas , \Xx \mapsto \big(\Pp (\Xx),\Sigma_w\big), T_{\Yy|\Xx} \mapsto  (T_{\Yy|\Xx})_*$, the {\it Lawvere functor}.
	
	(3) Let $T \in \Probm(\Xx, \Yy)$  and $\nu \ll \mu \in \Ss (\Xx)$. Then $T_*(\nu) \ll T_*(\mu)$ by  a result due  to 
	Ay-Jost-L\^e-Schwachh\"ofer \cite[Remark  5.4, p. 255]{AJLS17}, which generalizes Morse-Sacksteder's  result \cite[Proposition 5.1]{MS1966}, see also \cite[Theorem 2 (2)]{JLT2021}  for an alternative  proof.
	
	(4)  Let  $T_i  \in  \Probm (\Xx_i, \Xx_{i+1})$ for $i  = 1,2$.  Then we have  \cite[Lemma 5.5, p.69]{Chentsov72}
	\begin{equation}\label{eq:tcompose}
		\overline{T_2 \circ T_1} =  P_* (T_2)\circ \overline {T_1}.
	\end{equation}
\end{remark}

Next, we shall  study   a contravariant functor $\Ff_b$  from the category $\Probm$ to the  category $\mathbf{Ban}$.
For a measurable  space $\Xx$,  we denote by $\Ff_b (\Xx)$   the space of  measurable  bounded   functions  on  $\Xx$.  Then  $\Ff_b (\Xx)$ is a Banach space   with the  {\it sup norm} $\| f\|_\infty : = \sup_{x\in \Xx} | f(x)|$. 

The functor  $\Ff_b$ assigns to  each $\Xx \in \Probm$  the Banach space  $\Ff_b (\Xx)$   and   to every Markov kernel  $T_{\Yy|\Xx} $
  the pullback  $T^*_{\Yy|\Xx}$  which is defined  by  the following formula:
\begin{equation}\label{eq:markov*}
	(T_{\Yy|\Xx}^* f) (x): = \int_\Yy   f(y) \, d T_{\Yy|\Xx} (y |x)  \text{ for } f \in \Ff_b (\Yy)  \text{ and  } x \in \Xx.
\end{equation}	

Note  that if  $\kappa \in \Meas(\Xx, \Yy)$,   then   $(\underline{\delta\circ\kappa})^* = \kappa ^*$.  The following theorem  states that $T_{\Yy|\Xx}^* \in  \mathbf{Ban} \big(\Ff_b (\Yy), \Ff_b (\Xx)\big)$  and $\Ff_b$ is a  contravariant functor.

\begin{theorem}\label{thm:clbanach}  
	(1) $T_{\Yy|\Xx}^*$   maps  $\Ff_b (\Yy)$ to $\Ff_b (\Xx)$. It is a  bounded  linear  map between Banach  spaces $(\Ff_b (\Yy), \| \cdot\| _{\infty}), (\Ff_b (\Xx), \| \cdot\| _{\infty})$ of operator  norm 1, which  sends the characteristic  function $1_\Yy$ to  the characteristic  function $1_\Xx$.
	
	(2)    $T^*$ is a positive  operator, i.e.,  $T^*  f (x) \ge 0$   for all  $x \in \Xx$, if  $f(y) \ge 0  $ for all $y \in \Yy$.

	(3) For any $\mu \in \Mm (\Xx)$ and  $g \in \Ff_b (\Yy)$  we have
	\begin{equation}\label{eq:tpullback}
		\int_\Yy g\,d (T_{\Yy|\Xx})_* \mu  := \int_\Xx T_{\Yy|\Xx}^* (g) d\mu.
	\end{equation}

	(4) The   assignment  $\Ff_b: \mathbf{Probm}  \to \Ff_b \mathbf {Meas}, \Ff_b \Xx: = \Ff_b (\Xx), \, F T_{\Yy|\Xx}: = T_{\Yy|\Xx}^*,$  is a  contravariant  functor.
\end{theorem}
\begin{proof}  The first two assertions  of   Theorem \ref{thm:clbanach}  were proved by Chentsov  
	\cite[Corollary of Lemma 5.1 and its Corollary, p. 66]{Chentsov72}.
	
Let us prove the third assertion.	By  Definition of $(T_{\Yy|\Xx})_*\mu$, see   Equation \eqref{eq:Mhomomorphism}, for any $B \in \Sigma_\Yy$ %\eqref{eq:markov1S},  
	we have
	$$\int _\Yy  1_B\,  d (T_{\Yy|\Xx})_*\mu = \int_\Xx \int_\Yy 1_B\, d  T_{\Yy|\Xx} (\cdot |x)\, d\mu (x).$$
	By Definition of $T_{\Yy|\Xx}^*1_B$, see \eqref{eq:markov*},  we have
	$$\int_\Xx \int_\Yy 1_B\, d  T_{\Yy|\Xx} (\cdot |x)\,  d\mu (x) = \int_\Xx  T_{\Yy|\Xx}^*   (1_B)(x)\, d\mu (x).$$
	Comparing the  above identities, we  obtain     Equation \eqref{eq:tpullback}  for  $g = 1_B$.  Let  $\Ff_s (\Yy)$   denote  the space of simple functions  on $\Yy$.  Since the LHS and RHS  are  linear  in $g$,  the  equality  \eqref{eq:tpullback} is also valid for any $g \in \Ff_s (\Yy)$.  It is known that  $\Ff_s(\Yy)$ is dense  in $\Ff_b (\Yy)$ in the sup-norm,  see, for instance,  \cite[p. 66]{Chentsov1965}, taking into account  that $T^*_{\Yy|\Xx}$  is   a bounded linear map of norm 1,  we conclude    Assertion (3).
	
 To  prove  the last assertion of Theorem \ref{thm:clbanach},   it suffices to show that  for any $T_{\Zz|\Yy} \in \Probm (\Yy, \Zz)$, $T_{\Yy|\Xx} \in  \Probm (\Xx, \Yy)$  and $g \in \Ff_b (\Zz)$ we have
	\begin{equation}\label{eq:functo}
		(T_{\Zz|\Yy}\circ T_{\Yy|\Xx}) g =  T_{\Yy|\Xx}^*\circ T_{\Yy|\Zz}^* g.
	\end{equation} 
By  Assertion  (3),  we have  for any  $\mu \in \Mm(\Xx)$
\begin{align}
	\int _\Xx  (T_{\Zz|\Yy}\circ T_{\Yy|\Xx})^* g  d\mu  &= 
\int _\Zz  g\, d \big ((T_{\Zz|\Yy}\circ  T_{\Yy|\Xx})_*   \mu\big ) \nonumber\\
 \text{(by Proposition \ref{prop:chentsovbanach})}  & = \int _\Zz g\, d \big ((T_{\Zz|\Yy})_*\circ  (T_{\Yy|\Xx})_*   \mu\big )\nonumber\\
 \text{(by \eqref{eq:tpullback}) }  & = \int _\Zz T_{\Yy|\Xx}^* \circ T_{\Yy|\Zz}^*  g\, d\mu. \label{eq:functomu}
\end{align}
Replacing $\mu$  in  Equation  \eqref{eq:functomu} by $\delta_x $ for $x \in \Xx$ we obtain immediately  Equation \eqref{eq:functo}. 
\end{proof}

 \subsubsection{ A characterization of regular conditional probability measures}\label{subs:charreg}
 
 Using  the Lawvere functor $P$,   we shall     characterize      regular conditional probability   measures among probabilistic morphisms, using the concept of   the graph of a probabilistic morphism.
 Given  two  measurable  mappings
 $\overline T_i: \Xx \to \Pp(\Yy_i)$,  where i = 1,2,  let us consider the map $$\overline{T_1 \cdot T_2}: \Xx \to \Pp (\Yy_1 \times \Yy_2),\; x\mapsto   \mathfrak m^2 (\overline   T_1(x), \overline T_2(x)).$$
 The map $\overline{T_1 \cdot T_2}: \Xx \to \Pp (\Yy_1 \times \Yy_2)$  is
 a measurable mapping  since  it is  the  composition of  two measurable  mappings $(\overline T_1\times \overline T_2): \Xx  \to  \Pp (\Yy_1) \times \Pp (\Yy_2)$ and  $\mathfrak m^2: \Pp (\Yy_1) \times \Pp (\Yy_2) \to  \Pp (\Yy_1 \times \Yy_2)$.

 \begin{definition}\label{def:graph}  (1)  Given  two  probabilistic morphisms
 	$T_i: \Xx \leadsto \Yy_i$  for i = 1,2,     {\it the join  of  $T_1$ and $T_2$ }  is  the  probabilistic morphism  $T_1 \cdot T_2: \Xx \leadsto \Yy_1 \times \Yy_2$   whose  generating   mapping  is $\overline{T_1 \cdot T_2}: \Xx \to \Pp (\Yy_1 \times \Yy_2)$  given by
 	\begin{equation}\label{eq:joint}
 		\overline{T_1 \cdot T_2} (x): = \mathfrak m^2 (\overline   T_1(x), \overline T_2(x)).
 	\end{equation}

 	(2)  Given  a        probabilistic  morphism  $T: \Xx \leadsto \Yy$  we denote  the join  of  $\Id_\Xx$ with  $T$ by
 	$\Gamma  _T: \Xx \leadsto  \Xx \times  \Yy$  and call  it  the   {\it graph  of $T$}. 
 \end{definition}
 
 \begin{remark}\label{rem:joint} 
 	(1)	If  $\kappa: \Xx \to \Yy$  is  a measurable  mapping, regarded as a probabilistic morphism, then
 	the graph $\Gamma_\kappa: \Xx \leadsto \Xx \times \Yy$  is a  measurable mapping  defined  by $x \mapsto  (x, \kappa (x))$  for $x\in \Xx$, since  $\overline \Id_\Xx = \delta \circ  \Id_\Xx$, and  therefore $\mathfrak m^2  (\overline\Id_\Xx (x),\overline \kappa (x) )   =  \delta  _x \otimes \delta _{\kappa (x)} = \delta \circ \Gamma _{\kappa} (x)$.

 	(2)	 The notion of the graph of a probabilistic morphism  $f$ has been appeared  first in Jost-L\^e-Tran \cite{JLT2021}, the arXiv version,  but without a definition.  The first  definition of this  concept  has  been  given in  Fritz-Gonda-Perrone-Rischel's paper \cite{FGBR2020}, where  they call  the graph of  $f$  the {\it input-copy version} or {\it bloom} of $f$.  
 \end{remark}
 \begin{definition}[Almost surely equality]\label{def:asep}  \cite{Fritz2019}  Let $\mu \in \Pp(\Xx)$.   Two  measurable mappings $T, T':\Xx\to  \Pp(\Yy)$  will be called  {\it equal   $\mu$-a.e.} (with the shorthand notation $ T \stackrel{\mu\mhyp a.e.}{=}  T' $), if 
 	for  any $B \in \Sigma_\Yy$
 	$$\mu\{x\in \Xx: T(x)(B) \not =  T'(x)(B)\} = 0.$$
 \end{definition}

 \begin{theorem}[Characterization of regular conditional probability measures]\label{thm:marginal}
 	
 	\
 	
 	(1) A measurable mapping  $\overline T : \Xx \to \Pp (\Yy)$ is a regular conditional  probability measure for $\mu\in \Pp(\Xx \times \Yy)$   with  respect  to the    projection  $\Pi_\Xx$ if  and only if
 	\begin{equation}\label{eq:graph}
 		(\Gamma_T)_* \mu_\Xx =\mu.
 	\end{equation}
 	
 	(2)  If  $ \overline  T,\,  \overline T':\Xx \to \Pp (\Yy)$ are  regular conditional  probability measures for $\mu\in \Pp(\Xx \times \Yy)$   with  respect  to the    projection  $\Pi_\Xx$,  then  $  \overline  T \stackrel{\mu\mhyp a.e.}{=}  \overline T'$. Conversely,  if  $T$ is a  regular conditional  probability measure for $\mu\in \Pp(\Xx \times \Yy)$   with  respect  to the    projection  $\Pi_\Xx$  and  $  \overline  T \stackrel{\mu\mhyp a.e.}{=}  \overline T'$,  then $\overline T':\Xx \to \Pp (\Yy)$ is a regular conditional  probability measure for $\mu\in \Pp(\Xx \times \Yy)$   with  respect  to the    projection  $\Pi_\Xx$.
 \end{theorem}

\begin{proof}
    (1) To prove the first  assertion of  Theorem \ref{thm:marginal}  we   compute  the value  of  $(\Gamma _T)_* \mu_\Xx$   for $A \times B$ where $A \in \Sigma_\Xx$ and $B \in \Sigma_\Yy$
    \begin{equation} \label{eq:jointmeasure}
    (\Gamma_T)_* \mu_\Xx (A\times B) = \int_\Xx  \overline {\Gamma _T} (x) (A\times B)d\mu_\Xx(x) = \int _A \overline T (B|x)d\mu_\Xx.
   \end{equation}
Comparing with Equation \eqref{eq:regpr}, we conclude that  $(\Gamma_T)_*\mu_\Xx = \mu$ if and only if $\overline T: \Xx \to \Pp (\Yy)$ is a product  regular conditional   probability measure  for $\mu$.

(2)  Now   assume that  $ \overline  T,\,  \overline T':\Xx \to \Pp (\Yy)$ are  regular conditional  probability measures for $\mu\in \Pp(\Xx \times \Yy)$   with  respect  to the    projection  $\Pi_\Xx$.   Let $B \in \Sigma_\Yy$. Set $$X_B ^+: =   \{ x\in \Xx: \overline T (x)(B) - \overline {T' }(x) (B) > 0\}.$$
Since  $\overline {T}, \overline{T'} \in \mathbf{Meas} (\Xx, \Pp (\Yy))	$ the  set  $X_B ^+ $  is measurable.  By Equation \eqref{eq:graph},  we have
\begin{equation}\label{eq:pushgamma} 
\big ((\Gamma_T)_*\mu\big )(X_B^+\times B)- \big ((\Gamma_T)_*\mu\big)(X_B ^+\times B) = \int_{X_B ^+} (\overline {T}(x)- \overline {T'}(x)) (B)\, d\mu (x) = 0.
\end{equation}

It follows that $\mu (X_B^+) = 0$.  Similarly we   conclude  that
$$\mu \{x\in \Xx:  \overline{T} (x)(B) - \overline{T'} (x) (B)\le  0\}= 0.$$ 
This implies  that  $  \overline{T} \stackrel{\mu_\Xx\mhyp a.e.}{=} \overline{ T' }$.

 Let $\mu \in \Pp (\Xx)$ and $ \overline T \stackrel{\mu_\Xx\mhyp a.e.}{=} \overline T' $.   Using  Equation \eqref{eq:jointmeasure}, we conclude  the the   second assertion  in Theorem \ref{thm:marginal}(ii).
\end{proof}

In  the following lemma  we collect  several  useful   identities  concerning 
graphs of probabilistic morphisms.

\begin{lemma}\label{lem:decompo}  (1) Denote by $\Pi_\Yy$  the projection $\Xx \times \Yy \to \Yy$. Then for  any $T \in \Probm (\Xx, \Yy)$ we have
	\begin{eqnarray}
	T= \Pi_\Yy \circ  \Gamma_T \label{eq:tdecomp},\\
	\Pi_\Xx \circ \Gamma_T  = \Id _\Xx. \label{eq:inverse}
	\end{eqnarray}

(2)  Let  $p_1: \Xx \leadsto \Yy$ and $p_2: \Yy \leadsto  \Zz$ be  probabilistic morphisms. Then we have
\begin{equation}\label{eq:graphcomp} \Gamma_{p_2 \circ p_1} = (\Id_\Xx \times p_2) \circ \Gamma _{p_1}.
\end{equation}
$$
\xymatrix{ 
	&  \Xx \times \Yy \ar@{~>}[r]^{ (\Id_\Xx\times p_2)} & \Xx \times \Zz\\
	\Xx \ar@{~>}[ur]^{\Gamma_{p_1}} \ar@{~>}[urr]^{\Gamma_{p_2 \circ p_1}} \ar@{~>}[r]^{p_1} & \Yy\ar@{~>} [r]^{p_2} & \Zz
}
$$

(3) Let $\kappa \in \Meas (\Xx,\Yy)$  and $\pb \in \Probm (\Yy, \Zz)$. Then we have
\begin{equation}\label{eq:graphpusha}
	(\kappa \times \Id_\Zz)\circ  \Gamma_{\pb \circ \kappa}  =  \Gamma _{ \pb }\circ \kappa .	
\end{equation}
	\end{lemma}
\begin{proof} (1) For any $ x\in \Xx$ we have
$$ \overline {\Pi_\Yy \circ \Gamma _T} (x) \stackrel{\eqref{eq:tcompose}}{=} (\Pi_\Yy)_*  \overline{\Gamma_T} (x)\stackrel{\eqref{eq:joint}}{=} \overline T(x).$$
This proves \eqref{eq:tdecomp}. Similarly, we have
$$ \overline{\Pi_\Xx \circ \Gamma_T} (x) = \delta _x.$$
This  proves \eqref{eq:inverse}.

(2)    Using  \eqref{eq:tcompose},   for any $x \in \Xx$ we  have
$$\overline{(\Id_\Xx, p_2)\circ \Gamma _{p_1} } (x) =  (\Id_\Xx\times p_2)_*(\delta_x \cdot \overline{p_1}(x))$$
$$= \delta_x \cdot (p_2)_*\overline {p_1} (x) = \Gamma_{p_2 \circ p_1} (x).$$
 This proves  Equality  \eqref{eq:graphcomp}.

(3) By\eqref{eq:graphcomp},   we have
\begin{align}\label{eq:graphpusht}
	(\kappa \times \Id_\Zz) \circ  \Gamma_{\pb\circ \kappa}   &= 	(\kappa \times \Id_\Zz)\circ (\Id _\Xx \times  \pb)\circ  \Gamma _\kappa\nonumber\\ 
	& = (\kappa \times \pb) \circ \Gamma_\kappa\nonumber\\
	& = \Gamma _\pb  \circ \kappa.
\end{align}
This completes the proof of Lemma \ref{lem:decompo}.
\end{proof}
 
 \subsection{A Bayesian category  of  Bayesian inversions}
 Let  us first recall  the concept  of  a Bayesian statistical model, which plays a  key role in Bayesian statistics, see Remark \ref{rem:posterior}, and in our solution of    Problem SBI.
 
 \begin{definition}\label{def:bayesstat} \cite[Definition 5]{JLT2021}   A {\it Bayesian statistical model} \ is  a   quadruple $(\Theta, \mu_\Theta, \pb, \Xx)$,  where  $(\Theta,\mu_\Theta)$  is a probability space,   and  $\pb\in \Meas\big( \Theta , \Pp(\Xx)\big)$, which is  called a {\it sampling operator}.  The family of  probability measures $ \{\pb (\theta) \in \Pp  (\Xx), \theta \in \Theta\}$ is called {\it a parametric family of  sampling distributions of data $x\in \Xx$}   with parameter $\theta\in \Theta$.  
 	A Bayesian statistical model $(\Theta, \mu_\Theta, \pb, \Xx)$  is called {\it identifiable},  if  $\pb$ is an injective map, i.e.,  if the underlying statistical  model $(\Theta,\pb, \Xx)$  is identifiable.  A  Bayesian statistical model
 	$(\Theta, \mu_\Theta, \pb, \Xx)$ is called {\it dominated  by $\nu$}, if  $\nu$ is a $\sigma$-finite  measure  on $\Xx$ such that  for  any $\theta \in \Theta$ we have  $\pb (\theta) \ll \nu$.  The {\it predictive  distribution} $\mu_\Xx \in \Pp (\Xx)$  of    a    Bayesian statistical model $(\Theta, \mu_\Theta, \pb, \Xx)$ is defined as the {\it prior marginal probability} of  $x$, i.e., $\mu_\Xx : = (\Pi_\Xx)_*\mu$,  where $\mu: = (\Gamma_{\underline \pb})_*\mu_\Theta \in \Pp (\Theta\times \Xx) $ is the joint distribution of $\theta\in \Theta$ and $x\in \Xx$  whose regular conditional probability measure  with respect to the projection $\Pi_\Theta: \Theta \times \Xx \to \Theta$ is $\pb:\Theta \to \Pp(\Xx)$. 
 \end{definition}

\begin{remark}\label{rem:posterior}  In Bayesian statistics,   one    uses   a Bayesian statistical model $(
	\Theta, \mu_\Theta, \pb, \Xx)$  to    summarize   our knowledge about
	the  sampling  distribution $\pb^n:\Theta \to \Pp (\Xx^n)$ of data $ S_n \in \Xx^n$  conditioning  on parameter $\theta \in \Theta$. One {\it interprets}
	the value  of a   Bayesian  inversion $\qb^{(n)} (S_n) \in \Pp (\Theta)$ of $\pb^n$ relative to a {\it prior distribution} $\mu_\Theta \in \Pp (\Theta) $ as  the updated    value  of    our certainty about  parameter  $\theta$ after seeing   $S_n \in \Xx^n$ \cite{Schervish1997}, \cite{GV2017}.
\end{remark}

   Let $(\Theta, \mu_\Theta, \pb, \Xx)$ be a Bayesian statistical model. By Lemma  \ref{lem:decompo}, $\mu_\Xx = (\underline{\pb})_*\mu_\Theta$. Equivalently,  for   any $  A \in \Sigma_\Xx$ we have
  \begin{equation}\label{eq:priordis}
 	\mu_\Xx (A) = \int _\Theta  \pb(A|\theta)d\mu_\Theta (\theta). 
 \end{equation}   

We shall call   the map $\sigma_{\Xx, \Theta}: \Xx\times \Theta \to \Theta\times \Xx,  (x, \theta) \mapsto  (\theta, x),$ a {\it mirror  map}. 

 \begin{definition}\label{def:bayesinv}
 Let  $(\Theta, \mu_\Theta, \pb, \Xx)$ be a   Bayesian statistical model.
 A {\it Bayesian inversion} \index{Bayesian inversion} $ \qb: =\qb (\cdot \|\pb, \mu_\Theta)\in \Meas\big( \Xx, \Pp (\Theta)\big)$  of a Markov kernel  $\pb\in \Meas\big ( \Theta, \Pp (\Xx)\big )$   relative to  $\mu_\Theta$ is a Markov kernel  such that
 \begin{equation}\label{eq:bayesinv0}
 	(\sigma_{\Xx, \Theta})_*	(\Gamma_{\underline \qb})_*\mu_\Xx = (\Gamma_{\underline \pb})_* \mu_\Theta
 \end{equation}
 where  $\mu_\Xx = (\underline \pb)_* \mu_\Theta$  is the predictive distribution  of  the Bayesian model.
 \end{definition}

\begin{lemma}\label{lem:bsuff}  A  measurable mapping $\qb: \Xx \to \Pp (\Theta)$ is a  Bayesian inversion  of $\pb: \Theta \to \Pp (\Xx)$  relative to 
	$\mu_\Theta \in \Pp (\Theta)$ if and only for all $ B \in  \Sigma _\Theta$  we have
\begin{equation}\label{eq:bsuff}
	\qb (B|\cdot) = \frac{d (\underline \pb)_*  (1_B \mu_\Theta)}{d  (\underline \pb)_*  ( \mu_\Theta)} \in L^1 (\Xx, \pb_* \mu_\Theta).
\end{equation}
\end{lemma}
 Before giving   a proof of Lemma \ref{lem:bsuff}  we note that the RHS of  \eqref{eq:bsuff} is well-defined  by Remark \ref{rem:functors}.

\begin{proof}[Proof of Lemma \ref{lem:bsuff}]   First   we shall prove  the ``if" assertion.
	Assume  that    $\qb$ and $\pb$ satisfy \eqref{eq:bsuff}.  %By 
For $B \in \Sigma_\Theta$  and $ A \in \Sigma _\Xx$, we compute  % the value $\mu (A \otimes B)$, using \eqref{eq:jointmeasure}
\begin{equation} 
	(\Gamma_{\underline{\qb}})_*\big( (\underline {\pb} )_* \mu_\Theta\big) (B \times A) 	= \int_A\qb(B|x)\, d \big ((\underline{\pb})_* \mu_\Theta\big)(x).\label{eq:compo3}
\end{equation}

Plugging  \eqref{eq:bsuff} into  \eqref{eq:compo3}, we obtain 
\begin{align}
	(\Gamma_{\underline{\pb}})_* \big((\underline{\qb})_* \mu_\Xx\big) (B \times A) & = \int_A  \frac{ d(\underline \pb)_* (1_B \mu_\Theta)}{d({\underline \pb})_*\mu_\Theta}  d (\underline{\pb})_* \mu_\Theta\nonumber\\
	& = \int _A d(\underline \pb)_* (1_B \mu_\Theta)\nonumber \\
	 &=\int_\Theta   \pb (A|\theta)\, d1_B \mu_\Theta(\theta) \hspace{1cm} (\text{by } \eqref{eq:Mhomomorphism})\nonumber\\
&	=\big ((\Gamma _{\underline{\pb}})_*\mu_\Theta\big ) (A\times B).
	\label{eq:compo4}
\end{align}

From Equation \eqref{eq:compo4}  we conclude that  $\qb$ is  a Bayesian  inversion of $\pb$ relative to $\mu_\Theta$.
 
Next, we shall prove the ``only if"  assertion.    Assume that   $\qb$ is  a Bayesian  inversion of $\pb$ relative to $\mu_\Theta$.  Note that  for any  $A \in \Sigma_\Xx$ and $B \in \Sigma _\Theta$  we have
\begin{align}
	(\Gamma_{\underline \qb})_*\big ((\underline{\pb})_* \mu_\Theta\big) (A \times B) & =	\int_A\qb(B|x)\, d \big ((\underline{\pb})_* \mu_\Theta\big)(x), \label{eq:bsuff1}\\
	(\Gamma _{\underline \pb})_*\mu_\Theta (B \times A) &=  \int_A  \frac{ d(\underline \pb)_* (1_B \mu_\Theta)}{d({\underline \pb})_*\mu_\Theta}  d (\underline{\pb})_* \mu_\Theta. \label{eq:bsuff2}
\end{align}
Since  $\qb$ is  a Bayesian  inversion of $\pb$ relative to $\mu_\Theta$,  the RHS  of \eqref{eq:bsuff1} is equal to  the RHS  of \eqref{eq:bsuff2}.  
The same  argument as in the   proof  of  Theorem \ref{thm:marginal}(2) yields
Equation \eqref{eq:bsuff} immediately. This completes  the proof of Lemma \ref{lem:bsuff}.
\end{proof}
 
 \begin{proposition}\label{prop:bayesgroup}   (1)  If $\qb , \qb' \in \Meas (\Yy, \Pp (\Xx))$ are two Bayesian  inversions  of  $\pb: \Xx \to \Pp (\Yy)$ relative to  $\mu_\Xx \in \Pp (\Yy)$ then  $\qb \stackrel{(\underline{\qb})_* \mu_\Xx\mhyp a.e.} {=}\qb'$.
 	
 (2) Assume that  $\qb: \Yy \to \Pp (\Xx)$ is a Bayesian inversion of  $\pb: \Xx \to \Pp (\Yy)$ relative  to $\mu_\Xx \in \Pp (\Xx)$. Then  $\pb: \Xx \to \Pp (\Yy)$ is a Bayesian  inversion of $\qb$ relative  to  $(\underline{\pb})_* \mu_\Xx$.
 
 (3)  Let  $\underline{\qb_1} : \Xx_2 \leadsto \Xx_1$ be a  Bayesian inversion of
 $\underline{\pb_1}: \Xx_1 \leadsto \Xx_2$  relative to  $\mu_1 \in \Pp (\Xx_1)$. Assume that $\underline{\qb_2} : \Xx_2 \leadsto \Xx_3$ is a  Bayesian inversion of  $\underline{\pb_2}: \Xx_2 \leadsto \Xx_3 $ relative  to  $(\underline{\pb_2})_* \mu_1$.   Then  $ \underline{\qb_1} \circ\underline  {\qb_2}: \Xx_3 \leadsto \Xx_1$ is a Bayesian inversion  of  $\underline{\pb_2}\circ \underline {\pb_1}$ relative to  $\mu_1$.
 \end{proposition}
 
 \begin{proof}  The first assertion of Proposition \ref{prop:bayesgroup} is a direct consequence of Theorem \ref{thm:marginal} (2).
 	
 	The second assertion  follows directly from  Definition \ref{def:bayesinv}.
 
 Let us prove the last assertion of Proposition \ref{prop:bayesgroup}.   For  $A \in \Sigma_{\Xx_1}$,  we have
 \begin{align*}(\underline{\qb_2} \circ  \underline{\qb_1})_* (1_A\mu_1) &=  (\underline{\qb_2})_* \circ (\underline{\qb_1})_*  (1_A \mu_1) \hspace{3cm} \text{ (by}   \eqref{eq:Mhomomorphism}) \\
 	& = (\underline{\qb_2})_*  \big( \pb_1 ^* (1_A) (\underline{\qb_1})_* \mu\big)\hspace{2,7cm}  \text{ (by     Lemma \ref{lem:bsuff})}  )\\
 	& =  \underline{\pb_2}^* \circ \underline{\pb_1}^* (1_A) \big((\underline{\qb_2} \circ \underline{\qb_1})_* \mu \big)  \hspace{1,8cm}\text{ (by   Lemma \ref{lem:bsuff})}\\
 	& =  (\underline {\pb_1} \circ \underline {\pb_2})^*  (1_A)\big((\underline{\qb_2} \circ \underline{\qb_1})_* \mu \big)  \hspace{1,7cm}\text{ (by Theorem \ref{thm:clbanach}(4))}.
 \end{align*}
Taking into account Lemma \ref{lem:bsuff}, this completes  the  proof of Proposition \ref{prop:bayesgroup}.
 \end{proof}

  By Proposition \ref{prop:bayesgroup}  we can   form a  category   $\widetilde{{\bf Bayes}}$   whose  objects consists of  all  probability spaces $(\Xx, \mu_\Xx)$ and  whose  morphisms  from $ (\Xx, \mu_\Xx)$ to $(\Yy, \mu_\Yy)$   are  probabilistic morphisms   $T \in \Probm (\Xx, \Yy)$ such  that  $T_*\mu _\Xx = \mu_\Yy$ and   $T$ admits  a Bayesian inversion relative  to  $\mu_\Xx$.   
  	Denote by  $\mathbf {Bayes}$  the quotient category  of $\widetilde{\mathbf {Bayes}}$  by the  following equivalence relation:  Two morphisms $T, T' :  (\Xx, \mu_\Xx)
  \to (\Yy,\mu_\Yy)$ are equivalent if and only if $T_*\mu_\Xx =T'_* \mu_\Xx$.  Then  $\mathbf {Bayes}$  is a groupoid.

 \section{Bayesian supervised learning}\label{sec:bsl}

In this section, using the  results obtained in the previous section, we   propose   a mathematical model and its theoretical solution to  Problem
SBI under  the  following assumption.

\underline{Assumption (BSa)}. We regard the distribution $\Pp_x\in \Pp (\Yy)$  of a label $y\in \Yy$ of $x \in \Xx$ as the distribution of $y$ conditioning at $x$ and hence  can be expressed via   supervisor operators $h \in \Pp (\Yy)^\Xx$ as $\Pp _x : = h (x)$,  and formulate our  certainty   about $h$ as follows.  The space  $\Hh\subset \Pp (\Yy)^\Xx$ of  all possible supervisor operators that govern the   probability distribution of  labels $y$ conditioning  at $x$ is parameterized  by a measurable space $\Theta$  with a measurable mapping $\pb: \Theta \to \Pp (\Yy)^\Xx, \, \pb (\Theta) = \Hh$.  Our   certainty  about $h = \pb(\theta)$ is  encoded  in a  probability measure  $\mu_\Theta \in \Pp (\Theta)$. For any $X_m  = (x_1, \ldots, x_m) \in \Xx^m$  the joint distribution  $\Pp _{X_m}$ of labels   $(y_1, \ldots, y_m)\in \Yy^m$, where $y_i $ is  the label of $x_i$, $i \in [1, m]$, is governed 
by the  {\it evaluation  operator} $\Pi_{X_m}: \Pp (\Yy)^\Xx \supset \Hh \to  \Pp (\Yy^m)$ defined by:  
\begin{equation}\label{eq:evtm}
	\Pi_{X_m} (h) : = \big (h(x_1), \ldots,  h (x_m)\big)\in \Pp (\Yy)^m \stackrel{\frak m^m}{\to} \Pp (\Yy^m),
\end{equation}
 and, taking into account of  uncertainty  about $h$,
\begin{equation}\label{eq:joindistrb}
	\Pp_{X_m}:=	(\underline{\frak m ^m \circ  \Pi_{X_m}\circ \pb})_* \mu_\Theta \in \Pp (\Yy^m).
\end{equation}

\begin{remark}\label{rem:sbi}
	(1) $\Pp (\Yy) ^\Xx$  is a measurable space  with respect to the  cylindrical  $\sigma$-algebra $\Sigma _{cyl}$, i.e., the smallest  $\sigma$-algebra  on $\Pp (\Yy)^\Xx$  such that  for any $X_m \in \Xx^m$ the evaluation map $\Pi_{X_m}: \Pp (\Yy)^\Xx \to \Pp (\Yy)^m$ is measurable.

	(2)  The supervisor operator  $h \in \Pp (\Yy)^\Xx$ encodes the  probability of the label  $y(x)$ of $x$. The assumption that    observable labels  $ y(x)\in \Yy$  of $x\in \Xx$ are  conditionally  i.i.d.  with respect  to   the parameter $\theta \in \Theta$  is expressed  as follows.   For  any $m$  the  joint  distribution of the $m$-tuple  $(y(x_{1}), \ldots, y(x_{m}))\in \Yy^m$ is $ \otimes_{i =1}^m \Pi_{x_i} \pb (\theta) = \mathfrak m^m\circ \Pi_{X_m}\circ\pb (\theta)$ with  certainty  in $\theta$ expressed by $\mu_\Theta$, i.e., for   any $A_i \in \Sigma_\Yy, i\in [1, m]$, we have
	\begin{align*} \Pp_{X_m} (A_1 \times \ldots \times A_m) & = \int_\Theta   \Pi_{ i =1}^m \Pi_{x_i} \pb (A_i|\theta)d\mu_\Theta (\theta)\\
	 =\int_\Theta   \mathfrak m^m \circ\Pi_{X_m} \circ \pb (\Pi_{i =1}^m A_i|\theta)d\mu_\Theta (\theta) &= 	(\underline{\frak m ^m \circ  \Pi_{X_m}\circ \pb})_* \mu_\Theta (A_1 \times \ldots \times A_m).
	\end{align*}
	This equality says that the condition \eqref{eq:joindistrb}  expresses the  conditional  independency of  the labels $(y(x_1), \ldots, y(x_m))$.

	(3) If $\Xx$ consists  of a single  point, this  assumption is the classical assumption  of  the Bayesian  approach  to density/probability estimation  problem.  
\end{remark}

	 For  $S_n = \big ((x_1, y_1), \ldots , (x_n, y_n)\big) \in  (\Xx \times \Yy)^n$  we denote  by $\Pi_\Xx (S_n)$  the $\Xx^n$-component  of
$S_n$, namely  $\Pi_\Xx (S_n) =  (x_1, \ldots,  x_n)\in \Xx^n$. Similarly,  $\Pi_\Yy (S_n) = (y_1, \ldots, y_n)\in \Yy^n$.
Assumption (BSa) and Remark \ref{rem:sbi} lead to the following.

\begin{definition}\label{def:fposteriord}  A {\it   Bayesian learning  model for the supervised  inference problem SBI} consists of  a      quadruple $(\Theta, \mu_\Theta, \pb , \Pp (\Yy)^\Xx)$, where  $\mu_\Theta \in \Pp (\Theta)$ and $\pb: \Theta \to \Pp (\Yy)^\Xx$  is a   measurable mapping. 
	
	(1)  For any $X_m = (x_1, \ldots, x_m)\in  \Xx^m$, the  Bayesian
	statistical model $(\Theta, \mu_\Theta, \frak m^m\circ \Pi_{X_m}\circ \pb, \Yy^m)$
	parameterizes    sampling distributions  of  $Y_m = (y_1, \ldots, y_m)\in \Yy^m$, where $y_i$ is  a label of $x_i$, with the sampling operator $\frak m^m \circ \Pi_{X_m}\circ \pb: \Theta \to \Pp (\Yy^m)$.

	\begin{align}
		\xymatrix{
			&    (\Pp(\Yy)^\Xx,\pb_* \mu_\Theta)\ar[dr]^{\Pi_{X_m}}  &   & \\
			(\Theta, \mu_\Theta) \ar[ur]^{\pb}\ar[rr]^{\Pi_{X_m}\circ \pb} & &\Pp (\Yy)^m\ar@{^{(}->}[r]^{\frak m^m}  &  \Pp (\Yy^m).
		}\label{diag:super1}
	\end{align}
	(2) For  a training sample  $S_n\in (\Xx \times \Yy)^n$,  the  {\it posterior  distribution} $ \mu_{\Theta|S_n} \in \Pp (\Theta)$ after \index{Posterior distribution} seeing $S_n$ is  the value $ \qb^{(n)}_{\Pi_\Xx (S_n)}\big(\Pi_\Yy  (S_n)\big) $  of  a  Bayesian  inversion $\qb^{(n)}_{\Pi_\Xx (S_n)}:\Yy^n \to \Pp (\Theta) $  of  the     Markov kernel $\frak m^n \circ \Pi_{ \Pi_\Xx (S_n)}\circ \pb:\Theta \to \Pp (\Yy^n)$  relative  to  $\mu_\Theta$.
	
	(3) For $T_m = (t_1, \ldots , t_m) \in \Xx^m$, the {\it  predictive distribution}  $\Pp_{T_m}\in \Pp (\Yy^m) $ of  $m$-tuple  $(y_1, \ldots, y_m)$  where $y_i$ is the label of $t_i$, given a  training data   set $S_n \in (\Xx \times \Yy)^n$,
	is defined  as the  predictive  distribution  of  the    Bayesian statistical  model $(\Theta,   \mu_{\Theta|S_n}, \frak m^m\circ \Pi_{T_m}\circ \pb, \Yy^m)$, i.e., 
	\begin{equation}\label{eq:posteriorsuper}
		\Pp_{T_m}:= (\underline{\frak m^m\circ \Pi_{T_m}\circ \pb})_*  \mu_{\Theta|S_n} \in \Pp (\Yy^m).
	\end{equation}
	(4) The aim  of  a learner  is to  estimate  and approximate   the value  of the posterior   predictive    distribution 
	\index{Posterior predictive distribution} $\Pp_{T_m}$.
\end{definition}

\begin{remark}\label{rem:posteriors} Since  $\Pi_\Yy (S_n)$ are  labels  of  $\Pi_\Xx (S_n)$, after  seeing  $S_n\in (\Xx \times \Yy)^n$,  the value $\qb^{(n)}_{\Pi_{\Xx}(S_n)} (S_n)$ is the  posterior  distribution   of $\mu_\Theta$ after seeing  $S_n$.
\end{remark}
 
The following  Proposition  states that  we  can always  choose  $\Pp (\Yy)^\Xx$ as  a universal  parameter  space  for  solving  Problem SBI under  Assumption (BSa).
\begin{proposition}\label{prop:univ}
	 Let $\big(\Theta, \mu_\Theta, \pb, \Pp (\Yy)^\Xx\big)$  be a Bayesian supervised  learning  model.  Then $\big(\Pp (\Yy)^\Xx, \pb_* \mu_\Theta, \Id_{\Pp (\Yy)^\Xx}, \Pp (\Yy)^\Xx\big)$ is a Bayesian supervised  learning  model with  the following property.   Assume  that   $\qb^{(m)}: \Yy^m \to \Pp (\Theta)$  is  a  Bayesian  inversion of   $\frak m^m \circ \Pi_{ X_m}\circ \pb:\Theta \to \Pp (\Yy^m)$  for  $X_m \in \Xx^m$. Then  
	 $(\pb)_* \circ \qb^{(m)}: \Yy^m \to \Pp \big(\Pp (\Yy)^\Xx\big)$ is a Bayesian inversion  of $\frak m^m \circ \Pi_{X_m}: \Pp (\Yy)^\Xx \to \Pp (\Yy^m)$.
\end{proposition}
Proposition \ref{prop:univ}  is a direct  consequence  of the following  Lemma.
 
\begin{lemma}\label{lem:uni}  Let $(\Theta, \mu_\Theta, \pb, \Xx)$ be a  Bayesian statistical model  and $\kappa \in \Meas (\Theta, \Theta_1)$.  Assume that   the following  diagram
	$$	\xymatrix{ (\Theta, \mu_\Theta)\ar[d] ^{\kappa} \ar@{~>}[r]^{\underline \pb} &   (\Xx, \pb_* \mu_\Theta)\\
		(\Theta_1, \kappa_*\mu_\Theta)	\ar@{~>}
		[ur]^{\underline{\pb_1}}  &  
	}
	$$
	is commutative, i.e.,  $\underline \pb =  \underline \pb_1 \circ 
	\kappa$.   Let $\mu_{\Theta |\Xx}:  \Xx \to \Pp (\Theta)$ be  a Bayesian inversion  of $\pb$ relative to  $\mu_\Theta$. Then  $\kappa_* \circ \mu_{\Theta|\Xx}: \Xx \to \Pp (\Theta_1)$ is  a Bayesian inverse  of $\underline {\pb_1}$ relative to $\kappa _* (\mu_\Theta)$.
\end{lemma}

\begin{proof}[Proof of Lemma \ref{lem:uni}]
To prove that  $\kappa_*\circ \mu_{\Theta|\Xx}: \Xx \to \Pp (\Theta_1)$ is a  Bayesian inversion  of $\pb_1$ relative to  $ \kappa_*\mu_\Theta$, by  Equation \eqref{eq:bayesinv0} it suffices  to  show that 
\begin{equation}\label{eq:reduced2}
	(\sigma_{\Xx, \Theta})_* (\Gamma_{\kappa \circ \underline {\mu_{\Theta|\Xx}}})_*  (\underline {\pb_1} \circ \kappa)_* \mu _\Theta  = (\Gamma_{\underline \pb_1})_*  (\kappa_* \mu_\Theta).
\end{equation}

$$	\xymatrix@1{
	(\Theta, \mu_\Theta)\ar[d]^{\kappa} &   & \ar@/_1pc/@{~>}[ll]^{\underline {\mu_{\Theta|\Xx}}}(\Xx, (\pb_1 )_*\kappa_* \mu_\Theta)\\
	(\Theta_1, \kappa_*\mu_\Theta)	\ar@{~>}[urr]^{\underline{\pb_1}}  &   &
}
$$  
By the  formula \eqref{eq:graphcomp} for the graph of the composition of probabilistic morphism,  we have
\begin{equation}\label{eq:reduced3}
	(\Gamma_{\kappa \circ \underline {\mu_{\Theta|\Xx}}})_*  (\underline {\pb_1} \circ \kappa)_* \mu _\Theta =( \Id_\Xx \times \kappa )_* \circ (\Gamma _{\underline {\mu_{\Theta|\Xx}}})_* ( \underline {\pb_1} \circ  \kappa)_* \mu_\Theta.
\end{equation}
Since $\mu_{\Theta|\Xx}$ is  a Bayesian inversion of $\pb_1 \circ \kappa$ relative to $\mu_\Theta$, we obtain  from \eqref{eq:reduced3}, taking into account \eqref{eq:bayesinv0},
$$ (\sigma_{\Xx, \Theta})_* ( \Id_\Xx \times \kappa )_*\circ (\Gamma _{\underline {\mu_{\Theta|\Xx}}})_* ( \underline {\pb_1} \circ  \kappa)_* \mu_\Theta = (\kappa \times  \Id_\Xx)_* (\Gamma _{\underline{\pb_1} \circ \kappa})_*\mu_\Theta.$$
Using   the  formula \eqref{eq:graphpusha} for the graph of the composition of probabilistic morphism, this implies
$$(\kappa \times  \Id_\Xx)_* (\Gamma _{\underline{\pb_1} \circ \kappa})_*\mu_\Theta = (\Gamma_{\underline {\pb_1}})_*\kappa_* \mu_\Theta .$$
Combining the last two  equalities   with   \eqref{eq:reduced3}, this yields  the desired   Equation \eqref{eq:reduced2}.
\end{proof}

Next, we  need  to   find  a  formula  of a  Bayesian  inversion $\qb ^{(n)}$ of $\frak m ^n \circ \Pi_{X_n}: \Pp (\Yy)^{\Xx} \to \Pp (\Yy^n)$ relative to  a  prior  probability measure $\mu\in \Pp \big(\Pp (\Yy)^\Xx\big)$,  compute  the posterior  distribution $\qb ^{(n)} (Y_n)$  after   seeing  training data $S_n\in (\Xx \times \Yy)^n$  where $\Pi_\Xx (S_n) = X_n$ and $\Pi_{\Yy}(S_n) = Y_n$,  and  then   compute / approximate  the predictive    distribution  $(\mathfrak m^m \circ \Pi_{T_m} )_* (\qb ^{(n)} (Y_n)) \in \Pp (\Yy^m)$.  
Assume that  $X_n = (x_1, \ldots, x_n)\in \Xx^n$ and  $T_m  = (t_1,\dots,   t_m)\in \Xx^m$.  Let $A$ be the smallest subset       of $\Xx$ that contains  all $x_i , i \in [1, n]$, and all  elements $t_j , j \in [1, m] $.  Note that  the      restriction map  $R_A: \Pp (\Yy)^\Xx \to \Pp (\Yy)^A$ is measurable  with respect to the cylindrical $\sigma$-algebras.   Furthermore the following  diagram is commutative
$$
\xymatrix{
	\Pp (\Yy)^ \Xx \ar[rr]^{\frak m^n\circ \Pi_{X_n}}\ar[d] _{R _A} & & \Pp (\Yy^n)\\
	\Pp (\Yy)^{A}\ar[urr]_{\frak m^\circ \Pi^A_{X_n}} & &
}
$$

where  $\Pi_{X_n} ^A : \Pp (\Yy)^A \to \Pp (\Yy^n)$ is the evaluation  map  defined in \eqref{eq:evtm}  but for  $h \in \Pp (\Yy)^A$. As a corollary  of Lemma  \ref{lem:uni}, we obtain immediately the  following.

\begin{corollary}\label{cor:reducedf1} Let  $\qb ^{(n)} :  \Yy^n \to  \Pp (\Pp (\Yy)^\Xx)$ be a  Bayesian inversion  of $\frak m^n\circ \Pi_{X_n}: \Pp (\Yy^\Xx)\to \Pp (\Yy^n)$. Then 
	$(R_A)_* \circ  \qb ^{(n)} : \Yy^n \to \Pp (\Pp (\Yy)^A)$ is a Bayesian inversion   of $\frak m^n\circ \Pi^A_{X_n}: \Pp (\Yy)^A \to \Pp (\Yy^n)$.
\end{corollary}

If $\Yy$ is  a Polish space,   then $\Pp (\Yy)$ is a Polish  space. Hence, by  the Kolmogorov extension theorem,    any  probability  measure  on $\Pp (\Yy)^\Xx$ is  defined  uniquely  by  its restriction  to $\Pp (\Yy)^A$  where  $A$ is a  finite  subset  of $\Xx$.   We shall illustrate   Corollary \ref{cor:reducedf1}  in  popular    Gaussian  process  regressions in Baysian learning  for   Polish spaces $\Yy = \R^n$.

Gaussian  process  regressions in Baysian learning is a particular case  of  Bayesian regression  learning. In Bayesian   regression learning,\index{Bayesian regression}  we assume  that   label spaces $\Yy$ are  Euclidean   vector  spaces $V$ endowed  with   Borel $\sigma$-algebra, and the following  stochastic  relations  (BNR1), (BNR2) between input data $x\in \Xx$  and label data  $y \in \Yy= V$ hold.\\

(BNR1)   The   stochastic  relation between  $x$  and $y$  is    expressed     by a   function  $ f\in (V ^\Xx, \Sigma_{cyl})$  with  certainty encoded  in  a probability measure  $\mu \in  \Pp (V ^\Xx)$, and the measurement   of $f(x) \in V$  is corrupted  by an  independent  noise  $\eps \in  (V, \nu_\eps), \, \nu _\eps \in \Pp (V)$. In other words,  we have
\begin{equation}\label{eq:noise}
	y  =  f(x) + \eps :  \, f \in (V ^\Xx, \mu), \,\eps \in  (V, \nu_\eps) .
\end{equation}

(BNR2) For any  $X_m = (x_1, \ldots, x_m) \in \Xx^m$, the  distribution   of  the value  $\big(f(x_1), \ldots,  f(x_m)\big)$  is defined  by $ (\Pi_{X_m}^V)_* \mu$, where $\Pi_{X_m}^V: V^\Xx \to V^m$ is 
the canonical evaluation operator
$$\Pi_{X_m}^V ( f ) : =\big(f(x_1), \ldots,  f(x_m)\big) \text{  for } f \in V^\Xx.$$ 
Hence,   the  distribution  of  the measurement  $\big (y_1 = f(x_1)+ \eps, \ldots,   y_m = f (x_m)+ \eps\big)$  is $\big ((\Pi^V_{X_m})_* \mu\big) * \nu _\eps ^m$.
\begin{remark}\label{rem:bnr}
	(1) Equation \eqref{eq:noise}  means that the  probability   of  $y$  given  $x$ is $\delta_{f(x)} *\nu_\eps$. 
	
	(2)  The condition (BNR2) is motivated  by  the functoriality of the push-forward operator  $(\Pi_{X_m}^V)_*: \Pp (V^\Xx) \to \Pp (V^m)$. This condition is also supported by the Kolmogorov extension  theorem.
\end{remark}

In Corollary \ref{cor:bnrp} below,    we shall    show that  the assumptions (BNR1)   and (BNR2)  imply the assumption  (BSa).  Let  $ \pb^\eps: V^\Xx \to \Pp (V)^\Xx$  be defined as follows:
\begin{align}
	\pb^\eps (f) &: =   \delta_{f} *\nu_\eps\in \Pp (V)^\Xx,\nonumber\\
	(\delta_{f} *\nu_\eps ) (x) & : = \delta_{f(x)} * \nu_\eps \text{ for } x \in \Xx. \label{eq:noisep}
\end{align}

Now we define $\pb^0: V^\Xx \to \Pp(V)^\Xx$   by
\begin{align}
	\pb^0 (f) &:=   \delta_{f} \in \Pp (V)^\Xx,\nonumber\\
	\delta_{f}  (x)&: = \delta_{f(x)}   \text{ for } x \in \Xx.\label{eq:noise0}
\end{align}
\begin{lemma}\label{lem:measurable1}

	(1) For any  $X_m \in \Xx^m$   we have
	\begin{equation}\label{eq:bnrp}
		\Pi^V_{X_m} = \underline {\frak m ^m \circ \Pi_{X_m}\circ \pb ^0}.
	\end{equation}	
	
			(2) The  maps $\pb ^\eps$ and  $\pb^0$   are measurable. 
	
	%Consequently for  any $\mu \in V^\Xx$  and $X_m \in \Xx^m$   we have
	%\begin{equation}\label{eq:bnr}
	%(\Pi^V_{X_m})_*\mu = (\underline {\frak m ^m \circ \Pi_{X_m}\circ \pb ^0})_*\mu
	%\end{equation}
\end{lemma}

\begin{proof}	

	  (1)  Let us  first  show   that  $\pb ^0 : V^\Xx \to \Pp (V)  ^\Xx$ is measurable  with respect to the  cylindrical  $\sigma$-algebras $\Sigma_{cyl}$  on   $V^\Xx$ and on  $\Pp (V) ^\Xx$.    It suffices to show that  for  any  $X_m : = (x_1, \ldots, x_m)\in \Xx^m$  the composition  $ \Pi_{X_m} \circ \pb ^0: V^\Xx \to \Pp (V)^m $ is  measurable. 
	 Taking into account the identity
	$$ \Pi_{X_m} \circ \pb ^0 (f) = ( \delta_{ f(x_1)}, \ldots, \delta_{f(x_m)}) \in \Pp (V) ^m   ,$$
	the measurability  of $\Pi^V_{X_m} \circ \pb ^0$  follows the measurability of the Dirac map $\delta:  V \to \Pp (V),   v \mapsto \delta_v$.	Since
	\begin{equation}\label{eq:bnr}
		\Pi_{X_m} \circ \pb ^\eps (f) = ( \delta_{ f(x_1)}* \nu_\eps, \ldots, \delta_{f(x_m)}* \nu_\eps) \in \Pp (V) ^m  
	\end{equation}
	the measurability  of the map $\Pi_{X_m} \circ \pb ^\eps$  follows from  the measurability    of  $\pb^0$  and  the  convolution  map
	\begin{equation}\label{eq:convmu0}
		C_{\nu_\eps}: \Pp (V) \to \Pp  (V), \mu \mapsto \mu * \nu_\eps.
	\end{equation}
	which is measurable by Proposition \ref{prop:lawvera}(2b).
	Thus we  proved the measurability of $\pb^0$ and $\pb^\eps$.

	(1)  The first
	 assertion of Lemma \ref{lem:measurable1} is equivalent  to the commutativity  of the following diagram
\begin{equation*}%\label{eq:commpbv}
	\xymatrix{
		V^\Xx \ar [rr] ^{ \Pi^V_{X_m}}\ar[d]^{\pb^0} &&  V^m \ar[d]^{\delta}\\
		\Pp (V^\Xx)\ar [rr]^{ \frak m ^m \circ \Pi_{X_m}}& &\Pp (V^m)
	}
\end{equation*}
where  $\delta$ is the Dirac map :  $ \delta (x): =\delta _x$,   whose   measurability is well-known,  see Subsection \ref{subs:notation}.

\end{proof}

Since  $\otimes _{ i =1}^m (\delta_{f(x_i)} * \nu_\eps) = (\otimes_{ i =1}^m \delta_{f(x_i)}) * \nu_\eps ^m$ we obtain   immediately from  Lemma  \ref{lem:measurable1} the following
\begin{corollary}\label{cor:bnrp}   The assumptions  (BNR1)  and (BNR2)  imply  that for any $X_m = (x_1, \ldots, x_m)\in \Xx^m$ the distribution  of
	the  value  $\big (f(x_1), \ldots,  f(x_n)\big )\in  V^m$  is governed  by the   supervisor operator  $\pb ^0  (f)\in  \Pp (\Yy)^\Xx$    with  certainty encoded  in $\mu \in \Pp (V^\Xx)$  and the distribution  of
	the  value  $\big (f(x_1)+\eps, \ldots,  f(x_n)+\eps\big )\in  V^m$  is  governed by the supervisor operator $\pb^\eps(f)\in \Pp (\Yy)$ with certainty
	in $\mu \in \Pp (V^\Xx)$.
\end{corollary}

Lemma \ref{lem:measurable1} implies that    we can apply  the  Bayesian supervised  learning  model in  Definition \ref{def:fposteriord}  for learning the  probability of distribution  of     any $k$-tuple $(y_1, \ldots, y_k)$  where  $y_i = f(x_i)$ or  $y_i = f(x_i)+\eps$.

Taking into account Definition  \ref{def:fposteriord} of a Bayesian learning model for supervised inference problem SBI, Lemma \ref{lem:measurable1},  and Equations \eqref{eq:bnr}, \eqref{eq:convmu0}, we summarize  our discussion in the following.

\begin{definition}[{\bf Bayesian regression model}]\label{def:bnr} A  Bayesian  model  for learning      a function  $f: \Xx \to V$  from  corrupted  data $\{y_i = f(x_i) + \eps, i \in [1, n]\}$,  defined as in \eqref{eq:noise}, consists  of a   quadruple  $(\Theta, \mu_\Theta, h, V^\Xx)$ where  $(\Theta ,\mu_\Theta)$ is a    probability space,  and $h:\Theta \to V^\Xx$ is a measurable map.  
	
	\index{Bayesian regression}
	
	(1) For any $X_m = (x_1, \ldots, x_m)\in \Xx^m$ the Bayesian statistical model  $(\Theta, \mu_\Theta, C_{\nu_\eps ^m}\circ \delta \circ \Pi_{X_m}^V \circ h, V^m)$  is a  Bayesian  model  for  learning the probability distribution  of $(y_1=f(x_1)+\eps, \ldots, y_m=f(x_m)+ \eps)\in V^m$. 
	$$
	\xymatrix@1{
		\Theta \ar[r]^h & V^\Xx \ar@{~>}[rrr]^{\underline{C_{\nu_\eps ^m}\circ \delta \circ \Pi_{X_m}^V}} & & & V^m }.
	$$
	
	(2) For  training  data $S_n \in (\Xx \times V)^n$, the posterior  distribution $\mu_{\Theta|S_n} \in \Pp (\Theta)$  after  seeing   $S_n$  is the value $ \qb ^{(n, \eps)}_{\Pi_\Xx (S_n)} (\Pi_\Yy (S_n))$  of a Bayesian  inversion  $\qb ^{(n, \eps)}_{\Pi_\Xx (S_n)}: V^n \to \Pp (\Theta)$ of the Markov  kernel  $C_{\nu_\eps ^n}\circ \delta \circ \Pi_{\Pi_X (S_n)}^V \circ h : \Theta \to \Pp(V^n)$ relative to $\mu_\Theta$.
	
	$$
	\xymatrix@1{
		\Theta \ar[r]_h & V^\Xx \ar@/_1pc/@{~>}[rrr]_{\underline{C_{\nu_\eps ^n}\circ \delta \circ \Pi_{\Pi_X (S_n)}^V}} & & & V^n \ar@/_1pc/@{~>}[llll]_{\qb ^{(n, \eps)}_{\Pi_\Xx (S_n)} (\Pi_\Yy (S_n))} 
	}
	$$
	
	(3) For  new test data $T_m = (t_1, \ldots, t_m)\in \Xx^m$, the  posterior  predictive  distribution $\Pp_{T_m} \in \Pp (V^m)$ of   the tuple $\big(y_1 = f(t_1), \ldots,  y_m = f(t_m)\big)$     is  defined  as the  predictive  distribution  of the Bayesian  statistical model $(\Theta, \mu_{\Theta|S_n}, \frak m ^m \circ \Pi_{T_m} \circ \pb^0\circ h,\Pp(V^m)) $, i.e., \index{posterior predictive distribution}
	$$
	\xymatrix@1{
		(\Theta, \mu_{\Theta|S_n}) \ar[r]^h & V^\Xx \ar[rrr]^{\Pi_{T_m}^V} & & & V^m }
	$$
	\begin{equation}\label{eq:predictivenoise}
		\Pp_{T_m} =  (\Pi_{T_m}^V)_*  h_* \mu_{\Theta|S_n}\in \Pp (V^m).
	\end{equation}
	
\end{definition}

\begin{example}\label{ex:gauss} Letting  $\mu\in \Pp (\R^\Xx)$ in Definition \ref{def:bnr}  be    a Gaussian process and   $\nu_\eps :  =\Nn(0, \eps^2)$ - the standard centered  Gaussian measure on $\R$,  we   obtain   a Gaussian process  regression  in    statistical learning \cite{RW2006}, \cite{GV2017}.
 A  Gaussian  process $\Gg\Pp (m, K) \in \Pp (\R^\Xx)$ is defined  uniquely by its  mean  function $m : \Xx \to \R,  x \mapsto m \big(({\Pi _x})_* \Gg\Pp (m, K)\big)$, where 
 $ m \big(({\Pi _x})_* \Gg\Pp (m, K)\big)$ is the mean  of the  Gaussian measure $({\Pi _x})_* \Gg\Pp (m, K)$ on  $\R$  and by $K: \Xx \times \Xx\to \R$, the  covariance  function  of $ \Gg\Pp (m, K)$ \cite{Bogachev1998}. If   $K (x, x)> 0 $ for all  $x$  then  for  any $X_m \in \Xx^m$ its pushforward  measure  $(\Pi_{X_m} )_* \Gg\Pp (m, K)$ is a nondegenerate Gaussian measure on $\R^m$.  In this case  we have a   formula  for  posterior  distribution $\Gg\Pp (m, K)_{|S_n}$ of $\Gg\Pp (m, K)$ after seeing    training data  $S_n \in  (\Xx \times V)^n$,  and a formula for  the corresponding   posterior  predictive  distribution of   the  tuple  of labels  of $T_m =  (t_1, \ldots, t_m) \in \Xx^m$,  see, e.g.,  \cite{RW2006}, \cite{Bishop2006}. This formula    is obtained  by  a similar procedure      explained in  Corollary \ref{cor:reducedf1}, which  can be described  in the following   diagram  and  Lemma  \ref{lem:Bayesinreg} below.
 \begin{equation}\label{eq:Gaussproj}
 	\xymatrix{
 		& \big(\R^\Xx, \mu\big)\ar@{~>}[d]_{(\Pi_{T_m}^V \times \underline{\pb^{(n,\eps)}_{X_n}})} &\\
 		& \ar[dl]_{\Pi_{\R^m}}(\R^m \times \R^n, \mu_{m, n^\eps})\ar[dr]_{\Pi_{\R^n} }  & \\
 		\big(\R^m, (\Pi_{T_m}^V)_* \mu\big)   &      & \ar@{~>}[ll]_{(\Pi_{T_m}^V)\circ \underline{\qb^{(n,\eps)}_{X_n}}}\big(\R^n,(\underline{\pb_{X_n}^{(n, \eps)}})_*\mu\big).  \ar@{~>}[uul]_{  \underline{\qb^{(n, \eps)}_{X_n}} }
 	}
 \end{equation}
 Here  $X_n  = (x_1, \ldots , x_n) = \Pi_\Xx (S_n)$,
 $T_m =  (t_1, \ldots, t_m) \in \Xx^m$, $\mu = \Gg\Pp(m, K)$, $\pb_{X_n}^{(n, \eps)}: = \frak m ^n \circ \Pi_{X_n} \circ \pb ^\eps$,  $\mu_{m, n ^\eps} = \big(\Pi_{T_m}^V \times \underline{\pb_{X_n}^{(n, \eps)}}\big)_*\Gg\Pp(m, K)\in \Pp (\R^m \times \R^n)$  is also a  Gaussian measure on  $\R^{m +n}$,   and  $\qb^{(n, \eps)}_{X_n} :  V^k \to\Pp ( V^\Xx)$  is the Bayesian inversion  of $\pb^{(n, \eps)}_{X_n}:  \R^\Xx  \to \Pp (\R^k)$  obtained   by the classical   Bayes formula. 
 
 \begin{lemma}\label{lem:Bayesinreg} The measurable map  $ (\Pi^{V}_{T_m})_*   \qb ^{(n, \eps)}_{X_n}: \R^n \to \Pp  (\R^m)$  is a  product  regular   conditional probability measure  for $\mu_{m, n^\eps} :=\big(\Pi^V_{T_m}\times \underline{\pb ^{(n, \eps)}_{X_n}}\big)_* \mu$ with respect to the  projection $\Pi_{\R^n}: \R^m \times \R^n \to \R^n
 	$. Consequently, for    $(\underline{\pb ^{(n,\eps)}_{X_n}})_*\mu$-a.e.  $Y_n  \in \R^n$ the  posterior  predictive  distribution  $(\Pi_{X_m})_* \qb ^{(n, \eps)}_{X_n}  (Y_n)$ is  equal to $\qb (Y_n)$ where  $\qb: \R^n \to \Pp(\R^m)$ is  a regular  conditional  probability measure for $\mu_{m,n ^\eps}$   with respect to the projection $\Pi_{\R^m}: \R^m\times \R^n \to \R^m$. 
 \end{lemma}
 \begin{proof} 
 	
 	1) To prove the first assertion of  Lemma \ref{lem:Bayesinreg}, by Theorem \ref{thm:marginal}, it suffices  to prove the  following equality :
 	\begin{equation}\label{eq:reducebi}
 		(\sigma_{\R^m, \R^n})_*	(\Gamma_{\Pi^V_{T_m}\circ \underline{\qb ^{(n,\eps)} _{X_n}}})_* (\underline{\pb ^{(n, \eps)}_{X_n}})_*\mu  =  \big(\Pi^V_{T_m}\times \underline{\pb ^{(n, \eps)}_{X_n}}\big)_* \mu.
 	\end{equation}
 	The following diagram visualizes  the considered  probabilistic morphisms  and  probability  spaces. 
 	\begin{equation}\label{eq:reducebidiag}
 		\xymatrix{
 			(\R^\Xx, \mu)\ar[d]_{\Pi_{T_m} ^V} & & & \ar[lll] _{\Pi_{\R^\Xx}} (\R^\Xx \times \R^n, (\Gamma_{\underline{\pb ^{(n, \eps)}_{X_n}}})_* \mu)\ar[dl]_{\Pi_{\R^n}}\ar[ddl] ^{\Pi^V_{T_m} \times \Id_{\R^n}}\\
 			\R^m  & & \ar@{~>}[ull]_{\underline{ \qb ^{(n, \eps)}_{X_n}}}  \ar@{~>}[ll]_{\Pi^V_{T_m}\circ \underline{\qb ^{(n, \eps)}_{X_n}}}\R^n  & \\
 			& & \ar[ull]^{\Pi_{\R^m}} \Big (\R^m \times \R^n,\big(\Pi^V_{T_m}\times \underline{\pb ^{(n, \eps)}_{X_n}}\big)_* \mu\Big)  .\ar[u]^{\Pi_{\R^n}} 
 		}
 	\end{equation}

 	By \eqref{eq:graphcomp} we have
 	\begin{equation}\label{eq:reducebi1}
 		(\Gamma _{\Pi^V_{T_m}\circ \underline{\qb ^{(n,\eps)} _{X_n}}})_* (\underline{\pb ^{(n, \eps)}_{X_n}})_*\mu =  (\Id_{\R^n} \times \Pi^V_{T_m})_*  (\Gamma_{\underline{\qb ^{(n,\eps)} _{X_n}}})_*(\underline{\pb ^{(n, \eps)}_{X_n}})_*\mu.
 	\end{equation}
 	Since $\qb^{(n, \eps)}_{X_n}$ is a Bayesian  inversion of $\pb ^{(n,\eps)}_{X_n}$ relative  to $\mu$,  taking into account  Theorem \ref{thm:marginal}, we  obtain from \eqref{eq:reducebi1}
 	\begin{equation}\label{eq:reducebi2}
 		(\sigma_{\R^m, \R^n})_*		(\Gamma _{\Pi^V_{T_m}\circ \underline{\qb ^{(n,\eps)} _{X_n}}})_* (\underline{\pb ^{(n, \eps)}_{X_n}})_*\mu =  ( \Pi^V_{T_m }\times \Id_{\R^n} )_* (\Gamma_{\underline {\pb ^{(n,\eps)}_{X_n}}})_*\mu.
 	\end{equation}
 	By \eqref{eq:tdecomp}, we  have  
 	\begin{equation}\label{eq:reducebi3}\Id_{\R^n} \circ \Gamma_{\underline{\pb ^{(n,\eps)}_{X_n}}} =\underline{\pb ^{(n, \eps)} _{X_n}}.
 	\end{equation}
 	On the other hand, we have
 	\begin{equation}
 		\Pi^V_{T_m}\circ \Gamma_{\underline{\pb ^{(n,\eps)}_{X_n}}} =  \Pi^V_{T_m}\circ \Pi_{\R^\Xx}  \circ \Gamma _{\underline{\pb ^{(n, \eps)}_{X_n}}}\stackrel{\eqref{eq:inverse} }{ = }\Pi^V_{T_m}.\label{eq:reducebi4}
 	\end{equation}
 	From \eqref{eq:reducebi2}, \eqref{eq:reducebi3}, \eqref{eq:reducebi4}  we obtain \eqref{eq:reducebi}, and  the first assertion of Lemma \ref{lem:Bayesinreg}, immediately.
 	
 	2) The  second assertion of Lemma \ref{lem:Bayesinreg}  follows  from the first one taking into account the  almost surely uniqueness  of  Bayesian inversions  in the  case  dominated  Markov kernels.
 	
 \end{proof}
\end{example}
 \section{Final remarks}\label{sec:final}
 
 In this    paper, using  category theory of Markov kernels, we  proposed a  unified  model  of Bayesian  supervised  learning.       
 
   We  formulate  two open problems which are interesting  and  important  to the author.
   
 1)   For  interesting   spaces $\Xx$ and measurable space  $\Yy$   find       prior probability measures  on $\Pp (\Yy)^\Xx$  that  admit   feasible  computations  of    Bayesian inversions  of $\mathfrak m ^m4 \circ \Pi_{X_m}: \Pp (\Yy) ^\Xx \to \Pp (\Yy ^m)$  for  any $X_m \in \Xx^m$.
 
2)  Find a satisfying   notion of consistency of    a prior  distribution on   $\Pp (\Yy)^\Xx$, extending    a recent  result by Koerpernik-Pfaff \cite{KP2021}  concerning  consistency of Gaussian  process regression.
 
 Finally we note that Corollary \ref{cor:bnrp} remains valid  if we   relax  the condition in \eqref{eq:noise} such that in the condition (BNR2) the distribution of the measurement $(y_1 = f(x_1)+ \eps (x_1), \ldots, y_m = f(x_m)+ \eps(x_m) ) $ is  $(\Pi^V_{X_m})_* \mu)_* \otimes _{i =1}^m  \nu_\eps (x_i)$,  where $\eps(x_i) \in  (V, \nu (x_i))$ and $\nu (x_i) \in \Pp (V)$ for $i \in [1,m]$.
\section*{Acknowledgement}
The author wishes to   thank Arthur Parzygnat,  Tobias Fritz, Chafik Samir, Tien-Tam Tran for helpful discussions, and Domenico Fiorenza for  helpful  remarks and suggestions on an earlier  version of this paper.  A  part  of this  work  has been  conceived  while  she was  Visiting Professor  of Kyoto  University  from July till October 2022. She  is grateful to Kaoru Ono   and  the Research Institute of Mathematical Sciences  for their hospitality  and  excellent working  conditions in Kyoto.  The author also expresses her  gratitude to the  referee for his/her 
helpful comments and suggestions.

The research of HVL was supported by the Institute of Mathematics,  Czech Academy of Sciences (the institute grant Nr. RVO: 67985840) and GA\v CR  (the grant Nr.  GA22-00091S).

\end{document}